\documentclass[review]{elsarticle}

\usepackage{bm,amssymb,amsmath,mathrsfs}
\usepackage{amsthm}
\usepackage{lineno}
\usepackage[usenames]{color}

\usepackage{dcolumn}

\newcommand{\bq}{\begin{equation}}
\newcommand{\eq}{\end{equation}}
\newcommand{\bc}{\begin{center}}
\newcommand{\ec}{\end{center}}
\newcommand{\bit}{\begin{itemize}}
\newcommand{\eit}{\end{itemize}}
\newcommand{\ben}{\begin{enumerate}}
\newcommand{\een}{\end{enumerate}}

\theoremstyle{plain}
\newtheorem{theorem}{Theorem}[section]
\newtheorem*{theorem*}{Theorem}
\newtheorem{proposition}[theorem]{Proposition}
\newtheorem{lemma}[theorem]{Lemma}
\newtheorem{corollary}[theorem]{Corollary}
\newtheorem{remark}[theorem]{Remark}
\newtheorem{definition}[theorem]{Definition}
\newtheorem{conjecture}[theorem]{Conjecture}

\begin{document}

\journal{(internal report CC23-7)}

\begin{frontmatter}

\title{Analytical proofs for the properties of the probability mass function of the Poisson distribution of order $k$}

\author[cc]{S.~R.~Mane}
\ead{srmane001@gmail.com}
\address[cc]{Convergent Computing Inc., P.~O.~Box 561, Shoreham, NY 11786, USA}

\begin{abstract}
The Poisson distribution of order $k$ is a special case of a compound Poisson distribution.
For $k=1$ it is the standard Poisson distribution.
Our main result is a proof that for sufficiently small values of the rate parameter $\lambda$,
the probability mass function (pmf) decreases monotonically for all $n\ge k$
(it is known that the pmf increases strictly for $1\le n \le k$, for fixed $k\ge2$ and all $\lambda>0$).
The second main result is a partial proof that the difference (mean $-$ mode) does not exceed $k$.
The term `partial proof' signifies that the derivation is conditional on an assumption
which, although plausible and supported by numerical evidence, is as yet not proved.
This note also presents new inequalities, and sharper bounds for some published inequalities, for the Poisson distribution of order $k$.
\end{abstract}

\vskip 0.25in


\end{frontmatter}

\newpage
\setcounter{equation}{0}
\section{\label{sec:intro} Introduction}
The Poisson distribution of order $k$ is a special case of a compound Poisson distribution introduced by Adelson \cite{Adelson1966}.
The definition below is from \cite{KwonPhilippou}, with slight changes of notation.
\begin{definition}
  \label{def:pmf_Poisson_order_k}
  The Poisson distribution of order $k$ (where $k\ge1$ is an integer) and parameter $\lambda > 0$
  is an integer-valued statistical distribution with the probability mass function (pmf)
\bq
\label{eq:pmf_Poisson_order_k}
f_k(n;\lambda) = e^{-k\lambda}\sum_{n_1+2n_2+\dots+kn_k=n} \frac{\lambda^{n_1+\dots+n_k}}{n_1!\dots n_k!} \,, \qquad n=0,1,2\dots
\eq
\end{definition}
\noindent
For $k=1$ it is the standard Poisson distribution.
Several results about the structure of the pmf of the Poisson distribution of order $k$ are known.
For example, unlike the standard Poisson distribution,
for fixed $k\ge2$ and $\lambda>0$ the elements $f_k(n;\lambda)$, $n=1,\dots,k$ form a {\em strictly increasing} sequence (Lemma 1 in \cite{KwonPhilippou}).
The structure of the pmf of the Poisson distribution of order $k$ was mapped numerically in a recent note by the author \cite{Mane_Poisson_k_CC23_6}.
It was observed that the Poisson distribution of order $k\ge2$ can exhibit as many as {\em four} peaks simultaneously.
It was also stated in \cite{Mane_Poisson_k_CC23_6}, based on numerical observations but without proof,
that for sufficiently small values of the rate parameter $\lambda$, the pmf decreases monotonically for all values $n\ge k$.
Such an intuitive result should have a simple (and hopefully elegant) proof and should not rely on numerical simulations.
Our main result is a proof of the monotonic decrease for $n\ge k$, under conditions to be specified below.
The second main result is a partial proof that the difference (mean $-$ mode) does not exceed $k$.
The term `partial proof' signifies that the derivation is conditional on an assumption
which, although plausible and supported by numerical evidence, is as yet not proved.
In addition, some new inequalities, as well as sharper bounds for published inequalities, are presented.
The proofs in this note are analytical and do not rely on numerical simulations.

The structure of this paper is as follows.
Sec.~\ref{sec:notation} presents basic definitions and notation.
Sec.~\ref{sec:ineq} presents proofs of improvements to published inequalities and new inequalities.
Sec.~\ref{sec:pmf} presents a proof that the pmf decreases monotonically for all values $n\ge k$,
for fixed $k\ge2$ and a sufficiently small value of the rate parameter $\lambda>0$.
Sec.~\ref{sec:mode} presents a `partial proof' that the difference (mean $-$ mode) does not exceed $k$.
Sec.~\ref{sec:conc} concludes.

\newpage
\setcounter{equation}{0}
\section{\label{sec:notation}Basic notation and definitions}
For later reference we define $\kappa=k(k+1)/2$
and denote the mean by $\mu_k(\lambda)$ and the mode by $m_k(\lambda)$.
Instead of the probability mass function (pmf) $f_k(n;\lambda)$ in eq.~\eqref{eq:pmf_Poisson_order_k},
we follow \cite{KwonPhilippou} and work with $h_k(n;\lambda) = e^{k\lambda}f_k(n;\lambda)$.
The following facts will be useful below.
\begin{enumerate}
\item
Observe that $h_k(n;\lambda)$ is a polynomial in $\lambda$ with all positive coefficients.  
\item
For $n=0$ then $h_k(0;\lambda)=1$ for all values of $k$ and $\lambda$.
\item
For $n>0$ and $k\ge1$, $h_k(n;\lambda)$ has no constant term, so $h_k(n;0)=0$ and $h_k(n;\lambda)$ is strictly positive and increasing in $\lambda$ for $\lambda>0$.
\item
For $n>0$ and $k\ge1$, the polynomial $h_k(n;\lambda)$ has degree $n$ because the highest power of $\lambda$ which appears is given by the tuple $n_1=1$ and all the other $n_i$ are zero.
The corresponding term is $\lambda^n/n!$ and is the only term of this degree.
\item
For $k=1$, then $h_1(n;\lambda) = \lambda^n/n!$.
\item
For $n=1$, note that $h_k(1;\lambda)=\lambda$ for all $k\ge1$.
\item
For $n=2$, note that $h_k(2;\lambda) = \frac12\lambda^2 + \lambda$ for all $k\ge2$.
\end{enumerate}
The parameter $r_k$ was defined in \cite{KwonPhilippou} as the positive root of the equation $h_k(k;\lambda)=1$.
It was shown in \cite{KwonPhilippou} that $r_k$ is unique and $0 < r_k < 1$.
For later use below, we introduce the parameter $r_{k,n,c}$ as the unique positive root of the equation $h_k(n;\lambda)=c$ for $c>0$.
\begin{lemma}
For fixed $k\ge1$ and $n\ge1$, the equation $h_k(n;\lambda) = c$ has exactly one positive real root for $\lambda$, denoted by $r_{k,n,c}$.
\end{lemma}
\begin{proof}
We dispose of the case $k=1$ immediately because $h_1(n;\lambda) = \lambda^n/n!$ and the root is $\lambda=(cn!)^{1/n}$, which is clearly unique and positive.
We treat only $k\ge2$ below.
The proof employs the Intermediate Value Theorem.
Recall that for $k>0$ and $n>0$, $h_k(n;0)=0$ and $h_k(n;\lambda)$ is strictly positive and increasing in $\lambda$ for $\lambda>0$.
Also, because $h_k(n;\lambda)$ is a polynomial, its value is unbounded as $\lambda\to\infty$.
Hence, for fixed $k>0$ and $n>0$, there exists $\bar\lambda_{k,n,c}>0$ such that for $h_k(n;\lambda)>c$ for all $\lambda>\bar\lambda_{k,n,c}$.
Hence by the Intermediate Value Theorem, the equation $h_k(n;\lambda)=c$ has a root $r_{k,n,c}$ in the interval $r_{k,n,c}\in(0,\bar\lambda_{k,n,c}]$.
Because $h_k(n;\lambda)$ is strictly increasing in $\lambda$ for $\lambda \ge 0$, the root $r_{k,n,c}$ is unique.
\end{proof}
\begin{remark}
  We can estimate an upper bound for $r_{k,n,c}$.
  As noted above, the highest power of $\lambda$ in $h_k(n;\lambda)$ is given by the term $\lambda^n/n!$.
  Hence $h_k(n;\lambda) \ge \lambda^n/n!$ for $\lambda>0$.
  Set $\lambda^n/n!=c$ to deduce $r_{k,n,c} \le (cn!)^{1/n}$.
\end{remark}
\begin{remark}
  We can estimate a better upper bound for $r_{k,n,c}$ if $n$ is a multiple of $k$.
  Then the {\em lowest} power of $\lambda$ in $h_k(n;\lambda)$ is given by the tuple $n_k=n/k$ and the other $n_i$ are zero.
  The corresponding term in $h_k(n;\lambda)$ is $\lambda^{n/k}/(n/k)!$ and is the only term of this degree.
  Hence $h_k(k;\lambda) \ge \lambda^{n/k}/(n/k)!$ for $\lambda>0$.
  Set $\lambda^{n/k}/(n/k)!=c$ to deduce $r_{k,n,c} \le (c(n/k)!)^{k/n}$.
  In particular, if $n=k$ then $r_{k,n,c} \le c$.
\end{remark}
\begin{remark}
For $n=2$, then $h_k(2;\lambda) = \frac12\lambda^2 + \lambda$ for all $k\ge2$.
Hence we determine $r_{k,2,c}$ exactly by solving the quadratic equation $\frac12\lambda^2 + \lambda - c = 0$.
The solution (positive root) is 
\bq
\label{eq:r_k2c}
r_{k,2,c} = \sqrt{2c+1}-1 \,.
\eq
\end{remark}
\begin{remark}
Next consider the case $n=k>2$.
The term in $\lambda$ is $\lambda^1/1! = \lambda$.
The term in $\lambda^2$ is given by the tuples ($n_1=1$, $n_{k-1}=1$) and ($n_2=1$, $n_{k-2}=1$), etc.~(with all other $n_i$ zero).
Counting the factorial denominators yields
\bq
h_k(k;\lambda) = \lambda +\frac{k-1}{2}\lambda^2 + O(\lambda^3) \,.
\eq
Hence we can improve the upper bound for $r_{k,k,c}$ by solving for $\frac12(k-1)\lambda^2 + \lambda - c = 0$.
The positive root yields an improved upper bound for $r_{k,k,c}$ as follows
\bq
\label{eq:r_kkc_bound}
r_{k,k,c} < \frac{2c}{\sqrt{2c(k-1)+1}+1} \,.
\eq
\end{remark}
\noindent
The most important case below is $c=2$ so we define $t_k=r_{k,k,2}$ (for `two') to avoid clumsy notation.
For most of the rest of this note, we shall hold $k\ge2$ and $\lambda>0$ fixed and vary only the value of $n$.
For brevity of the exposition, we adopt the notation by Kostadinova and Minkova \cite{KostadinovaMinkova2013}
and write ``$p_n$'' in place of $h_k(n;\lambda)$ and mostly omit explicit mention of $k$ and $\lambda$. 
The recurrence for $p_n$ is as follows (from eq.~(6) in \cite{KwonPhilippou}, terms with negative indices are set to zero).
\bq
\label{eq:KP_rec_pn}
p_n = \frac{\lambda}{n} \,\sum_{j=1}^k jp_{n-j} \,.
\eq
Kostadinova and Minkova also published the following recurrence, which contain four terms for any values of $n$ and $k$.
(Proposition 1 in \cite{KostadinovaMinkova2013}, terms with negative indices are set to zero.)
\bq
\label{eq:KMrec}
p_n = \Bigl(2 + \frac{\lambda-2}{n}\Bigr)p_{n-1}
-\Bigl(1 -\frac{2}{n}\Bigr)p_{n-2}
-\frac{k+1}{n}\,\lambda p_{n-k-1}
+\frac{k}{n}\,\lambda p_{n-k-2} \,.
\eq
The following two expressions will be useful below.
Again, terms with negative indices are set to zero.
\begin{enumerate}
\item
Using eq.~\eqref{eq:KP_rec_pn} yields
\bq
\label{eq:KP_pn1_minus_pn}
\begin{split}
p_{n+1} - p_n &= \frac{\lambda}{n+1} \biggl(\sum_{j=1}^k jp_{n+1-j}\biggr) - \frac{\lambda}{n} \biggr(\sum_{j=1}^k jp_{n-j}\biggr)
\\
&= \frac{\lambda}{n(n+1)} \biggl\{\, np_n +(n-1)p_{n-1} +\dots +(n-k+1)p_{n-k+1} \,\biggr\} -\frac{k}{n}\,\lambda p_{n-k} 
\\
&= \frac{\lambda}{n(n+1)} \biggl(\sum_{j=0}^{k-1} (n-j)p_{n-j}\biggr) - \frac{k}{n}\,\lambda p_{n-k} \,.
\end{split}
\eq
\item
Using eq.~\eqref{eq:KMrec} yields 
\bq
\label{eq:KM_pn1_minus_pn}
p_n - p_{n-1} = \frac{\lambda}{n}p_{n-1} + \frac{n-2}{n}(p_{n-1}-p_{n-2})
-\frac{k+1}{n}\,\lambda p_{n-k-1}
+\frac{k}{n}\,\lambda p_{n-k-2} \,.
\eq
\end{enumerate}

\newpage
\setcounter{equation}{0}
\section{\label{sec:ineq}Inequalities}
\subsection{Alternative proof of Lemma 1 in \cite{KwonPhilippou}}
\begin{lemma}
\label{lemma:KP_Lemma1_restatement}  
(Restatement of Lemma 1 in Kwon and Philippou \cite{KwonPhilippou}, with notation employed in this note.)
For $2 \le n \le k$ and a fixed $\lambda>0$,
\bq  
\lambda \le p_{n-1} < p_n \,.
\eq
\end{lemma}
\begin{proof}
Consider $n \in [2,k]$, so $p_{n-k-1}=p_{n-k-2}=0$ in eq.~\eqref{eq:KM_pn1_minus_pn}.
We proceed by induction on $n$.
First suppose that $p_{n-1}-p_{n-2}>0$. Then eq.~\eqref{eq:KM_pn1_minus_pn} yields
\bq
\label{eq:KM_pn1_minus_pn_diff} 
p_n - p_{n-1} = \frac{\lambda}{n}p_{n-1} + \frac{n-2}{n}(p_{n-1}-p_{n-2}) > 0\,.
\eq
The right hand side is positive because $n\ge2$ and $p_{n-1}>0$ for $\lambda>0$.
Next note that $p_2-p_1 = (\lambda/2) p_1 = \frac12\lambda^2 > 0$ for $\lambda>0$.
Hence the result follows by induction on $n$.
Also $p_n > p_1=\lambda$.
\end{proof}
\begin{remark}
We could also prove the result using eq.~\eqref{eq:KP_pn1_minus_pn} because $p_{n-k}=0$ for $n<k$.
Hence $p_{n+1}-p_n>0$ because it equals a sum of positive terms for all the relevant values of $n$ in Lemma \ref{lemma:KP_Lemma1_restatement}.
\end{remark}
\begin{remark}
We obtain the following strictly increasing sequence for fixed $k\ge2$, $\lambda>0$ and $n=1,\dots,k$.
\bq
\label{eq:KP_incseq_pn}  
\lambda = p_1 < p_2 < \dots < p_k \,.
\eq
\end{remark}

\subsection{Improved upper bound for $r_k$}
Recall that Kwon and Philippou \cite{KwonPhilippou} defined $r_k$ as the unique positive root of $h_k(k;\lambda)=1$ and they proved that $0 < r_k < 1$.
We can improve the upper bound as follows.
Set $c=1$ in eq.~\eqref{eq:r_kkc_bound} to obtain
\bq
\label{eq:rk_new_upbound} 
r_k < \frac{2}{\sqrt{2k-1}+1} \,.
\eq
For $k=1$, eq.~\eqref{eq:rk_new_upbound} yields the exact answer $r_1=1$.
For $k=2$, eq.~\eqref{eq:rk_new_upbound} also yields the exact answer $r_2 = \sqrt{3}-1$ (see eq.~\eqref{eq:r_k2c} with $c=1$).
For $k>2$, eq.~\eqref{eq:rk_new_upbound} yields an upper bound.

\subsection{Improved bound for Lemma 3 in \cite{KwonPhilippou}}
\begin{lemma}
(Restatement of Lemma 3 in Kwon and Philippou \cite{KwonPhilippou}, with notation employed in this note.)
For $k \ge 2$ and $0 < \lambda \le r_k < 1$,
\bq
\label{eq:pk_gt_pk1}
p_k > p_{k+1} \,.
\eq
\end{lemma}
\noindent
Note the following: set $n=k$ in eq.~\eqref{eq:KP_pn1_minus_pn}, then $p_{n-k}=p_0=1$ and we obtain
\bq
\label{eq:KPLemma3_diff}
p_{k+1} - p_k = \frac{\lambda}{k(k+1)} \biggl(\sum_{j=1}^k jp_j\biggr) - \lambda \,.
\eq
Next from eq.~\eqref{eq:KP_incseq_pn}, $p_j < p_k$ for $j=1,\dots k-1$, which yields the inequality
\bq
\label{eq:pk1_minus_pk}
\begin{split}
p_{k+1} - p_k &< \frac{\lambda p_k}{k(k+1)} \biggl(\sum_{j=1}^k j\biggr) - \lambda 
\\
&= \lambda\biggl(\frac{p_k}{2} - 1 \biggr) \,.
\end{split}
\eq
Kwon and Philippou \cite{KwonPhilippou} set $\lambda=r_k$ in their Lemma 3,
whence $p_k=1$ so the rhs in eq.~\eqref{eq:pk1_minus_pk} is negative
\bq
p_{k+1} - p_k < r_k\Bigl(\frac{1}{2} - 1 \Bigr) = -\frac{r_k}{2} < 0 \,.
\eq
Hence $\lambda \le r_k$ is a sufficient but not necessary bound.
It is possible to do better.
\begin{proposition}
\label{prop:tk}
For fixed $k \ge 2$ and $0 < \lambda \le t_k$, i.e.~$p_k \le 2$, then $p_k > p_{k+1}$.
\end{proposition}
\begin{proof}
Set $p_k\le2$ in eq.~\eqref{eq:pk1_minus_pk}, then $p_{k+1} - p_k < 0$.
Exactly at $p_k=2$, the value of $\lambda$ is $t_k$.
\end{proof}
\begin{remark}
\label{remark:rkk2_upbound}
The constraint $p_k\le2$ (equivalently $\lambda\le t_k$) is also sufficient but not necessary.
The upper bound for $t_k$ itself is obtained by setting $c=2$ in eq.~\eqref{eq:r_kkc_bound}, viz.
\bq
t_k \le \frac{4}{\sqrt{4k-3}+1} \,.
\eq
\end{remark}
\begin{remark}
  Observe that the combination of eqs.~\eqref{eq:KP_incseq_pn} and \eqref{eq:pk_gt_pk1} and Prop.~\ref{prop:tk} imply that
  for fixed $k\ge2$ and $0 < \lambda \le t_k$, the point at $n=k$ is a local maximum of the pmf of the Poisson distribution of order $k$.
  This does {\em not} necessarily imply that it is a mode (global maximum).
  See the results in \cite{KwonPhilippou,PhilippouFibQ,Mane_Poisson_k_CC23_3}, when the point at $n=k$ is a mode.
  Note that the results in \cite{Mane_Poisson_k_CC23_3} are observations from numerical calculations.
\end{remark}

\subsection{\label{sec:KPLemma3a}Bound to obtain $p_{k+1}\ge p_k$}
Let us consider the opposite inequality $p_{k+1} \ge p_k$.
This is the domain when the point at $n=k$ ceases to be a local maximum.
\begin{proposition}
  For fixed $k\ge2$ and $\lambda \ge 2$, then $p_{k+1} \ge p_k$.
\end{proposition}
\begin{proof}
Recall eq.~\eqref{eq:KPLemma3_diff}.
Now we say that $p_j \ge p_1(=\lambda)$ for $j=1,\dots k-1$, which yields the opposite inequality
\bq
\begin{split}
p_{k+1} - p_k &\ge \frac{\lambda^2}{k(k+1)} \biggl(\sum_{j=1}^k j\biggr) - \lambda 
\\
&= \frac{\lambda^2}{2} - \lambda \,.
\end{split}
\eq
Hence $p_{k+1} \ge p_k$ if $\lambda \ge 2$.
\end{proof}
\begin{remark}
The constraint $\lambda\ge2$ is a sufficient but not necessary lower bound.
It is a tight bound in the sense that it is attained for $k=1$ but it is not tight for $k\ge2$.
\end{remark}
\begin{corollary}
Once again, we can do better by employing $p_2=\frac12\lambda^2+\lambda$.
Then $p_{k+1} \ge p_k$ for fixed $k\ge2$ and $\lambda \ge q_k$, where
\bq
\label{eq:qk_def}
q_k = \frac{4}{\sqrt{5-4/\kappa}+1} \,.
\eq
Note that $\sqrt{5}-1 < q_k \le (\sqrt{33}-3)/2$.
\end{corollary}
\begin{proof}
We employ $p_2=\frac12\lambda^2+\lambda$.
Then we obtain the inequality
\bq
\begin{split}
  p_{k+1} - p_k &\ge \frac{\lambda}{k(k+1)}\biggl[(2+\dots+k)p_2 + p_1\,\biggr] -\lambda
\\
&= \frac{\lambda}{k(k+1)}\biggl[(2+\dots+k)\Bigl(\frac{\lambda^2}{2}+\lambda\Bigr) + \lambda\,\biggr] -\lambda
\\
&= \frac{\lambda}{k(k+1)}\biggl[(2+\dots+k)\frac{\lambda^2}{2} + (1+\dots+k)\lambda\,\biggr] -\lambda
\\
&= \frac{\lambda}{k(k+1)}\biggl[\frac{k(k+1)-2}{2}\frac{\lambda^2}{2} + \frac{k(k+1)\lambda}{2}\,\biggr] -\lambda
\\
&= \lambda\,\biggl[\frac{k(k+1)-2}{4k(k+1)}\lambda^2 + \frac{\lambda}{2} -1\,\biggr] \,.
\end{split}
\eq
To obtain the upper bound we set $p_{k+1}-p_k=0$ and solve the following quadratic equation.
\bq
(k(k+1)-2)\lambda^2 + 2k(k+1)\lambda -4k(k+1) = 0 \,.
\eq
Writing $k(k+1)=2\kappa$ yields a more concise equation
\bq
(\kappa-1)\lambda^2 + 2\kappa\lambda -4\kappa = 0 \,.
\eq
The term in $\lambda^2$ is positive for $k\ge2$, hence it lowers the root to $\lambda<2$.
The quadratic equation has two real roots of opposite sign. The positive root is (for $k\ge2$)
\bq
\begin{split}
\lambda &= \frac{\sqrt{\kappa(5\kappa-4)}-\kappa}{\kappa-1}
= \frac{4\kappa}{\sqrt{\kappa(5\kappa-4)}+\kappa} \,.
\end{split}
\eq
Elementary manipulations yield eq.~\eqref{eq:qk_def}.
The upper bound on $q_k$ is obtained by setting $k=2$ and the lower bound is obtained by taking the limit $k\to\infty$.
\end{proof}
\begin{remark}
  This is also a sufficient but not necessary lower bound.
  For large $k\gg1$ the asymptote for the bound is 
\bq  
q_k \to \sqrt{5} - 1 \simeq 1.236068 \,.
\eq
A graph of $q_k$ is plotted in Fig.~\ref{fig:graph_pk1_ge_pk} for $2 \le k \le 100$.
The asymptote $\sqrt{5}-1$ is plotted as the dashed line.
The value of $q_k$ decreases monotonically to the asymptote as $k$ increases.
\end{remark}

\newpage
\setcounter{equation}{0}
\section{\label{sec:pmf}Probability mass function}
The structure of the probability mass function (pmf) of the Poisson distribution of order $k$ was mapped in \cite{Mane_Poisson_k_CC23_6}, in a set of histogram plots.
It was stated in \cite{Mane_Poisson_k_CC23_6}, based on numerical observations but without proof,
that for sufficiently small values of the rate parameter $\lambda$, the pmf decreases monotonically for all values $n\ge k$.
Such an intuitive result should have a simple (and hopefully elegant) proof and should not rely on numerical simulations.
We treat $p_n = h_k(n;\lambda)$ below.
This section presents an analytical proof that for fixed $k\ge2$ and sufficiently small $\lambda>0$,
the value of $p_n$ decreases monotonically for $n\ge k$.

First we dispose of the case $k=1$.
It is well-known that for $k=1$, the pmf decreases monotonically with $n$ for all $\lambda<1$.
It is not so simple for $k\ge2$, because, as proved in \cite{KwonPhilippou} and noted several times already,
the points $p_n$ for $n=1,\dots,k$ always form a strictly increasing sequence, for all $\lambda>0$.
This increasing sequence does not exist in the standard Poisson distribution (because it is the single point $n=1$, hence not a `sequence').
Our attention below is therefore for $k\ge2$ and the points $n\ge k$.

By ``sufficiently small $\lambda>0$'' we mean that for fixed $k\ge2$,
there exists an open neighborhood of zero, {\em whose size depends on $k$ only},
and all the statements in the proof below are true if $\lambda$ lies in this neighborhood (and $\lambda>0$).
Here is an informal discussion to determine a ``sufficiently small'' value for $\lambda$.
\begin{enumerate}
\item
  For any fixed $k\ge2$, there are at most $n^k/k!$ tuples in the sum for $p_n$.
\item
  For $\lambda<1$, their sum never exceeds $n^k/k!$.
\item
  Hence if we choose $\lambda < k!/n^k$, the magnitude of the term in $\lambda^i$ will exceed that of the term in $\lambda^{i+1}$,
  for every power $i$ which appears in the sum of tuples in $p_n$.
\item
  Of course this upper bound on $\lambda$ depends on $n$ (as well as $k$), hence a more refined upper bound is required below.
\end{enumerate}
We formulate the overall proof as follows.
We first formulate an induction proof, conditional on the existence of a starting block of elements with properties to be specified below.
We then prove the existence of a starting block of elements with the requisite properties, to complete the induction proof.

\newpage
\begin{proposition}
\label{prop:inductionproof}(Induction proof)
For fixed $k\ge2$, and $n \ge 2k$, suppose that there exists a fixed $\lambda>0$ such that the block of $k+1$ contiguous elements
$\{p_{n-k},p_{n-k+1},\dots,p_n\}$ form a strictly decreasing sequence $p_{n-k} > p_{n-k+1} > \dots > p_n$.
Then $p_{n+1}-p_n < 0$, i.e.~the sequence can be extended to include $p_n > p_{n+1}$.
\end{proposition}
\begin{proof}
We employ eq.~\eqref{eq:KP_pn1_minus_pn}.
By hypothesis, $p_{n-j} < p_{n-k}$ for all $j\in[0,k-1]$, hence 
\bq
\begin{split}  
p_{n+1} - p_n &< \frac{\lambda p_{n-k}}{n(n+1)} \Bigl[nk -(1+\dots+(k-1))\Bigr] - \frac{k}{n}\,\lambda p_{n-k} 
\\
&= \frac{2n -(k-1) -2(n+1)}{2(n+1)}\,\frac{k}{n}\lambda p_{n-k} 
\\
&= -\frac{k+1}{2(n+1)}\,\frac{k}{n}\lambda p_{n-k}
\\
&< 0 \,.
\end{split}  
\eq
It remains to prove the existence of a starting block of $k+1$ elements with the requisite properties.
\end{proof}
\begin{proposition}
\label{lemma:block_k_2k}
For fixed $k\ge2$, there exists a fixed $\lambda>0$ such that the block of $k+1$ contiguous elements
$\{p_k,\dots,p_{2k}\}$ form a strictly decreasing sequence.
Define the upper bound
\bq
\label{eq:lam_monotone_upbound}
\lambda_k^< = \min\bigl\{t_k,\,k!/(2k)^k\bigr\} \,.
\eq
Then $p_k > p_{k+1} > \dots > p_{2k}$ for $0 < \lambda \le \lambda_k^<$. 
\end{proposition}
\begin{proof}
  The proof proceeds in several steps.
\item
  For $n=k+1,\dots,2k$, the lowest power of $\lambda$ which appears in $p_n$ is $\lambda^2$.
\item
  For $n=k+j$, where $j=1,\dots,k$, the term in $\lambda^2$ is given by the sum of the tuples with the values
  ($n_j=n_k=1$ and all other $n_i$ zero, i.e.~$n=j+k$), ($n_{j+1}=n_{k-1}=1$ and all other $n_i$ zero, i.e.~$n=(j+1)+(k-1)$), etc.
\item
  For $j\in[1,k-1]$, there are $\lfloor(k+1-j)/2\rfloor$ such tuples and for $j=k$ there is exactly one tuple $(0,\dots,0,2)$.
\item
  Then $p_{k+j} = \frac12(k+1-j)\lambda^2 + O(\lambda^3)$ for $j\in[1,k]$.
\item
  {\em Hence the value of the coefficient of $\lambda^2$ decreases by $\frac12$ as $j$ increases in unit steps, for $j=[1,k]$.}
\item
  Hence for sufficiently small $\lambda>0$, $p_{k+j} > p_{k+j+1}$ for all $j=[1,\dots,k-1]$.  
\item
  In this context, for ``sufficiently small'' we can set $\lambda < \inf\{k!/n^k, n\in[k+1,2k]\} = k!/(2k)^k$.
\item
  We moreover already know that $p_k>p_{k+1}$ for $0<\lambda \le t_k$.
\item
  Define the upper bound $\lambda_k^< = \min\bigl\{t_k,\,k!/(2k)^k\bigr\}$.
\item
  Hence we can restrict the value of $\lambda$ to the interval
  $0 < \lambda < \lambda_k^<$.
\item
  This is probably a sufficient but not necessary upper bound.
\item
  Then the elements $\{p_k,\dots,p_{2k}\}$ form a strictly decreasing sequence $p_k > p_{k+1} > \dots > p_{2k}$.
\end{proof}
\noindent
Note that the upper bound $\lambda_k^<$ depends only on $k$ and {\em not} on $n$.
Hence we have a suitable starting block of $k+1$ elements, with an upper bound for $\lambda$,
and the overall proof follows from Prop.~\ref{prop:inductionproof} by induction on $n$.

\newpage
\begin{remark}
  Numerical calculations indicate that for $2\le k \le 10^4$,
  the upper bound value $\lambda_k^< = 2/(k+1)$ suffices to yield a monotonically decreasing pmf for all $n\ge k$.
  This bound implies the value of the mean is $\mu_k(\lambda)=\kappa\lambda\le k$.
  Intuitively this makes sense because it was proved in \cite{GeorghiouPhilippouSaghafi} that the mode is less than the mean,
  and it was proved in \cite{Mane_Poisson_k_CC23_5} that if the mode is nonzero, its value must be at least $k$,
  hence if the mean is less than $k$ the mode must be zero.
  In such a circumstance one might expect that the pmf decreases monotonically for all $n\ge k$.
  The numerical calculations also indicate that $2/(k+1)$ is a sufficient but not necessary upper bound.
  The determination of a more optimal value for the upper bound $\lambda_k^<$ is a matter for future research.
\end{remark}
\begin{remark}
  Numerical calculations indicate the optimal value for the upper bound $\lambda_k^<$ might be given by solving for the positive real root of the equation $p_{k+1}=p_{k+2}$.
  This is the value of $\lambda$ at which the two histogram bins at $k+1$ and $k+2$ are equal and form a ``shoulder'' in the histogram of the pmf for $n > k$.
  Numerical calculations find no exceptions of monotonicity for $2\le k \le 10^4$.
  Fig.~\ref{fig:graph_hist_k4} displays a histogram plot of $p_n$ for $k=4$ and $\lambda=0.6026076$.
  The points at $n=k+1=5$ and $n=k+2=6$ have equal height, to within numerical precision.
  The histogram decreases monotonically (or is nonincreasing) for $n \ge k(=4)$. 
  Nevertheless, this claim for the optimal upper bound should be regarded as preliminary.
\end{remark}

\newpage
\setcounter{equation}{0}
\section{\label{sec:mode}Mode}
The following upper and lower bounds for the mode have been proved (Theorem 2.1 in \cite{GeorghiouPhilippouSaghafi})
\bq  
\label{eq:mode_Georghiou_etal_Thm2.1}
\bigl\lfloor \kappa\lambda \bigr\rfloor - \kappa + 1 - \delta_{k,1} \le m_k(\lambda) \le \bigl\lfloor \kappa\lambda \bigr\rfloor \,.
\eq
The upper bound in eq.~\eqref{eq:mode_Georghiou_etal_Thm2.1} is attained, hence sharp.
Numerical studies reported in \cite{Mane_Poisson_k_CC23_6} led to the following conjecture for an improved lower bound for the mode,
for cases where the value of the mode is nonzero.
\begin{conjecture}
\label{conj:mode_low_bound}
For fixed $k\ge2$ and $\lambda$ sufficiently large so the value of the mode is nonzero, the mode is bounded below as follows.
\bq
\label{eq:non_asymp_mode_lower_bound}
m_k(\lambda) \ge \lfloor\kappa\lambda\rfloor -k \,.
\eq
It was shown in \cite{Mane_Poisson_k_CC23_6} that eq.~\eqref{eq:non_asymp_mode_lower_bound} is attained and is hence a sharp lower bound.
It was proved (Prop.~(2.4) in \cite{Mane_Poisson_k_CC23_6}) that if the mode is nonzero then $\lfloor\kappa\lambda\rfloor -k \ge 0$,
hence the right-hand side in eq.~\eqref{eq:non_asymp_mode_lower_bound} is never negative.
\end{conjecture}
\noindent
It was proved in \cite{Mane_Poisson_k_CC23_5} that if the mode is nonzero, its value must be at least $k$, i.e. $m_k(\lambda) \ge k$.
It was shown that $m_k(\lambda) = k$
(i) for $k=2$ in \cite{PhilippouFibQ},
(ii) for $k=2,3,4$ in \cite{KwonPhilippou} and
(iii) for $2 \le k \le 14$ in \cite{Mane_Poisson_k_CC23_3}.
The findings in \cite{Mane_Poisson_k_CC23_3} were based on numerical calculations, which also indicated that $m_k(\lambda) > k$ for all tested values $k\ge15$.

Here we offer a partial proof of Conjecture \ref{conj:mode_low_bound}.
The term `partial proof' signifies that the derivation is conditional on an assumption which,
although plausible and supported by numerical evidence, is as yet not proved, as will be explained below.

We first fix $k\ge2$ and choose an integer $n$ such that $n-k$ is the (nonzero) value of the mode.
The fact that the mode is nonzero (and also $m_k(\lambda)\ge k$) necessarily implies $n-k\ge k$, i.e.~$n\ge 2k$.
{\em We suppose for now that the mode is unique.}
We shall discuss the case of a bimodal distribution below.
Next comes a key unproved assumption:
\begin{quote}
{\em We assume that the pmf is nonincreasing from $n-k$ through $k$, i.e.~$p_{n-k} \ge p_{n-k+1} \ge \dots \ge p_n$.}
\end{quote}
The above assumption is supported by numerical evidence but is as yet not proved, which is why
eq.~\eqref{eq:non_asymp_mode_lower_bound} retains its status as a conjecture.

Given all of the above, by assumption $p_j \ge p_n$ for $j=n-k,\dots,n-1$. 
We then process the recurrence for $p_n$ in eq.~\eqref{eq:KP_rec_pn} as follows.
\bq  
\label{eq:mean_mode_dist3}
\begin{split}
np_n &= \lambda\,\bigl(p_{n-1} +2p_{n-2} +\dots + kp_{n-k}\bigr)
\\
&\ge \lambda p_n\,(1+\dots+k)
\\
&= \frac{k(k+1)}{2}\,\lambda p_n 
\\
&= \kappa\lambda p_n \,.
\end{split}
\eq
Cancelling $p_n$ and noting that the mean is $\mu_k = \kappa\lambda$ (derived in \cite{PhilippouMeanVar}) yields the inequality $\mu_k(\lambda) \le n$.
Since the mode is $m_k(\lambda) = n-k$, it follows that the difference between the mean and the mode does not exceed $k$.
\bq
\label{eq:mean_mode_dist5}
\mu_k(\lambda) - m_k(\lambda) \le n - (n-k) = k \,.
\eq
We reexpress this as a lower bound for the mode.
Since the mode is always an integer, we employ the floor function to derive the following lower bound for the mode
\bq
\label{eq:mode_low_bound_ge}
m_k(\lambda) \ge \lfloor\kappa\lambda\rfloor - k \,.
\eq
This is the bound in eq.~\eqref{eq:non_asymp_mode_lower_bound}.
We now discuss the caveats in the derivation of eq.~\eqref{eq:mode_low_bound_ge}.
\begin{enumerate}
\item
  It was shown in \cite{Mane_Poisson_k_CC23_6} that equality in eq.~\eqref{eq:mode_low_bound_ge} is attained.
  The example offered was $k=2$ and $\kappa\lambda=4$, hence $\lambda=4/3$ (Fig.~9 in \cite{Mane_Poisson_k_CC23_6}).
  Fig.~\ref{fig:graph_hist_k2} displays a histogram plot of $p_n$ for $k=2$ and $\lambda=4/3$.
  The mode is $2$ and the lower bound is $\lfloor\kappa\lambda\rfloor-k = \lfloor4\rfloor-2 = 2$,
  hence eq.~\eqref{eq:mode_low_bound_ge} is satisfied.
  However, the sequence of $k+1$ points $\{p_2,p_3,p_4\}$ is {\em not} nonincreasing.
  This demonstrates that the assumption of a nonincreasing sequence of $k+1$ points $\{p_{n-k},\dots,p_n\}$ (with the mode at $n-k$) is open to challenge.
\item
  The value of $\lambda$ was increased to $\lambda=4.02373/3$ and the resulting histogram plot of $p_n$ is displayed in Fig.~\ref{fig:graph_hist_k2_bimodal}.
  The histogram is bimodal, to within numerical precision, with joint modes at $2$ and $4$.
  The value of the lower bound is $\lfloor\kappa\lambda\rfloor = \lfloor 4.02373\rfloor -2 = 2$.
  The mode value of $4$ satisfies eq.~\eqref{eq:mode_low_bound_ge} with strict inequality
  while the mode value of $2$ satisfies eq.~\eqref{eq:mode_low_bound_ge} with equality.
  The derivation of eq.~\eqref{eq:mode_low_bound_ge} may possibly be valid for a bimodal distribution,
  but it is ambiguous which point to select as ``the mode'' in the derivation.
\item
  It was shown in \cite{Mane_Poisson_k_CC23_3} that for any $k\ge2$, the Poisson distribution of order $k$ has a denumerable infinity of
  double modes consisting of pairs of consecutive integers.
  Then eq.~\eqref{eq:mode_low_bound_ge} works, because the sequence of $k+1$ points $\{p_{n-k},\dots,p_n\}$ is nonincreasing, with $n-k$ selected to be the lower mode value.
\item
  Observe that equality is attained in eq.~\eqref{eq:mode_low_bound_ge} only if {\em all} the numbers in the sequence $\{p_{n-k},\dots,p_n\}$ are equal.
  If even one pair of elements in the sequence exhibits a strict decrease, $p_j > p_{j+1}$, eq.~\eqref{eq:mode_low_bound_ge} becomes a strict inequality
\bq
\label{eq:mode_low_bound_gt}
m_k(\lambda) > \lfloor\kappa\lambda\rfloor - k \,.
\eq
Numerical evidence indicates there are {\em no} instances of three or more consecutive equal values of the histogram bins $p_n$ (for $\lambda>0$), but to date this is not proved.
\item
  We explain the need for the mode to be nonzero.
  If the (unique) mode is zero, it has been remarked several times already that the sequence $\{p_1,\dots,p_k\}$ is {\em strictly increasing}
  for any fixed $k\ge2$ and $\lambda>0$.
  This invalidates the derivation of eq.~\eqref{eq:mode_low_bound_ge}.
\item
  For the same reason, we require $n\ge2k$, to avoid including any points in the interval $\{p_1,\dots,p_{k-1}\}$, because their presence
  invalidates the derivation of eq.~\eqref{eq:mode_low_bound_ge}.
  It was proved in \cite{Mane_Poisson_k_CC23_5} that if the mode is nonzero, its value is at least $k$, hence $n\ge2k$.  
\end{enumerate}
Given all of the above, we can revise Conjecture \ref{conj:mode_low_bound} as follows.
\begin{conjecture}
\label{conj:mode_low_bound_update}
For fixed $k\ge2$ and $\lambda$ sufficiently large so the value of the mode is nonzero, the mode is bounded below as follows.
\begin{enumerate}
\item
  If the mode is unique, the mode is bounded strictly as follows
\bq
\label{eq:non_asymp_mode_lower_bound_gt}
m_k(\lambda) > \lfloor\kappa\lambda\rfloor -k \,.
\eq
\item
  If the distribution is bimodal with the modes at a pair of consecutive integers, eq.~\eqref{eq:non_asymp_mode_lower_bound_gt} is applicable for both mode values.
\item
  If the distribution is bimodal with modes at nonconsecutive integers,
  the higher mode value is bounded by eq.~\eqref{eq:non_asymp_mode_lower_bound_gt} 
  and the lower mode value is bounded as follows
\bq
\label{eq:non_asymp_mode_lower_bound_ge}
m_k(\lambda) \ge \lfloor\kappa\lambda\rfloor -k \,.
\eq
\end{enumerate}
It was shown in \cite{Mane_Poisson_k_CC23_6} that eq.~\eqref{eq:non_asymp_mode_lower_bound_ge} is attained and is hence a sharp lower bound.
It was proved (Prop.~(2.4) in \cite{Mane_Poisson_k_CC23_6}) that if the mode is nonzero then $\lfloor\kappa\lambda\rfloor -k \ge 0$,
hence the right-hand sides in eqs.~\eqref{eq:non_asymp_mode_lower_bound_gt} and \eqref{eq:non_asymp_mode_lower_bound_ge} are never negative.
Numerical evidence in \cite{Mane_Poisson_k_CC23_6} indicates the Poisson distribution of order $k$ does not have three or more joint modes.
\end{conjecture}
\begin{remark}
  The minimum value of $\lambda$ for the mode to be nonzero is not known precisely, although a few cases have been solved.
  The case $k=2$ was solved in \cite{PhilippouFibQ} and the cases $k=3$ and $4$ were solved in \cite{KwonPhilippou}.
  See also upper and lower bounds in \cite{Mane_Poisson_k_CC23_6} and numerical results in \cite{Mane_Poisson_k_CC23_3}.
\end{remark}
\begin{remark}
  It is reasonable to suppose that once the index of the unique mode has been passed, or that of the higher mode value in the case of a bimodal distribution,
  the histogram bins (values of $p_n$) of the probability mass function decrease monotonically, or at least do not increase, but there is as yet no proof of this.
\end{remark}
\begin{remark}
  \label{remark:single_pt_pn_min}
  Note that Conjecture \ref{conj:mode_low_bound_update} only requires the value of the single point $p_n$ not exceed the values of any the previous $k$ points
  $\{p_{n-k},\dots,p_{n-1}\}$. Those $k$ points need not form a nonincreasing sequence.
  All we require to derive eq.~\eqref{eq:mean_mode_dist3} is $p_n \le \min(p_{n-k},\dots,p_{n-1})$.
\end{remark}
\begin{remark}
  Remark \ref{remark:single_pt_pn_min} carries the consequence that if there is any sequence of $k+1$ points $\{p_{n-k},\dots,p_n\}$
  where $p_n$ has the least value, i.e.~$p_n \le \min(p_{n-k},\dots,p_{n-1})$, then the mean cannot exceed $n$.
  Since it has been proved that the mean is never less than the mode \cite{GeorghiouPhilippouSaghafi},
  this implies that there is no such sequence in the pmf until the value of $n$ is at least one unit larger than the index of a (non-unique) mode: $n \ge m_k(\lambda)+1$.
  It is a matter for future research.
\end{remark}

\newpage
\section{\label{sec:conc}Conclusion}
The major item of this note was in Sec.~\ref{sec:pmf}, a proof that the probability mass function (pmf) of the Poisson distribution of order $k$
decreases monotonically for all $n\ge k$, for fixed $k\ge2$ and a sufficiently small value of the rate parameter $\lambda>0$.
(For $1 \le n \le k$, it has been proved \cite{KwonPhilippou} that the pmf increases strictly, for all $k\ge2$ and $\lambda>0$.)
The second main result was in Sec.~\ref{sec:mode}, a partial proof that the difference (mean $-$ mode) does not exceed $k$.
The term `partial proof' was employed because the derivation is conditional on an assumption which,
although plausible and supported by numerical evidence, is as yet not proved.
In addition, several improvements to published inequalities were proved (sharper bounds, etc.) and also some new inequalities.


\newpage
\begin{figure}[!htb]
\centering
\includegraphics[width=0.75\textwidth]{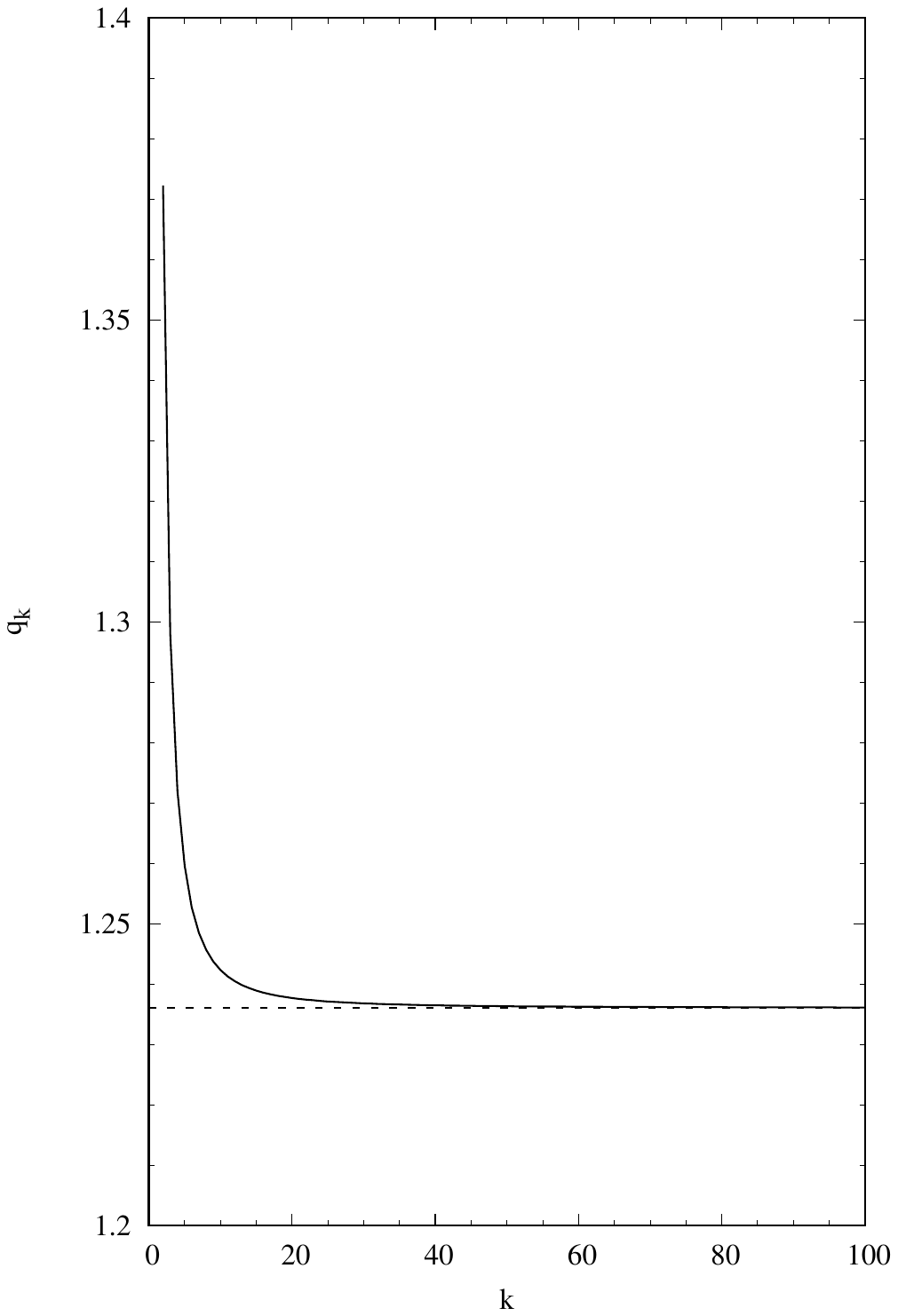}
\caption{\small
  \label{fig:graph_pk1_ge_pk}
  Graph of $q_k$ for the Poisson distribution of order $k$ for $2 \le k \le 100$.
  The asymptote $\sqrt{5}-1$ is plotted as the dashed line.}
\end{figure}

\newpage
\begin{figure}[!htb]
\centering
\includegraphics[width=0.75\textwidth]{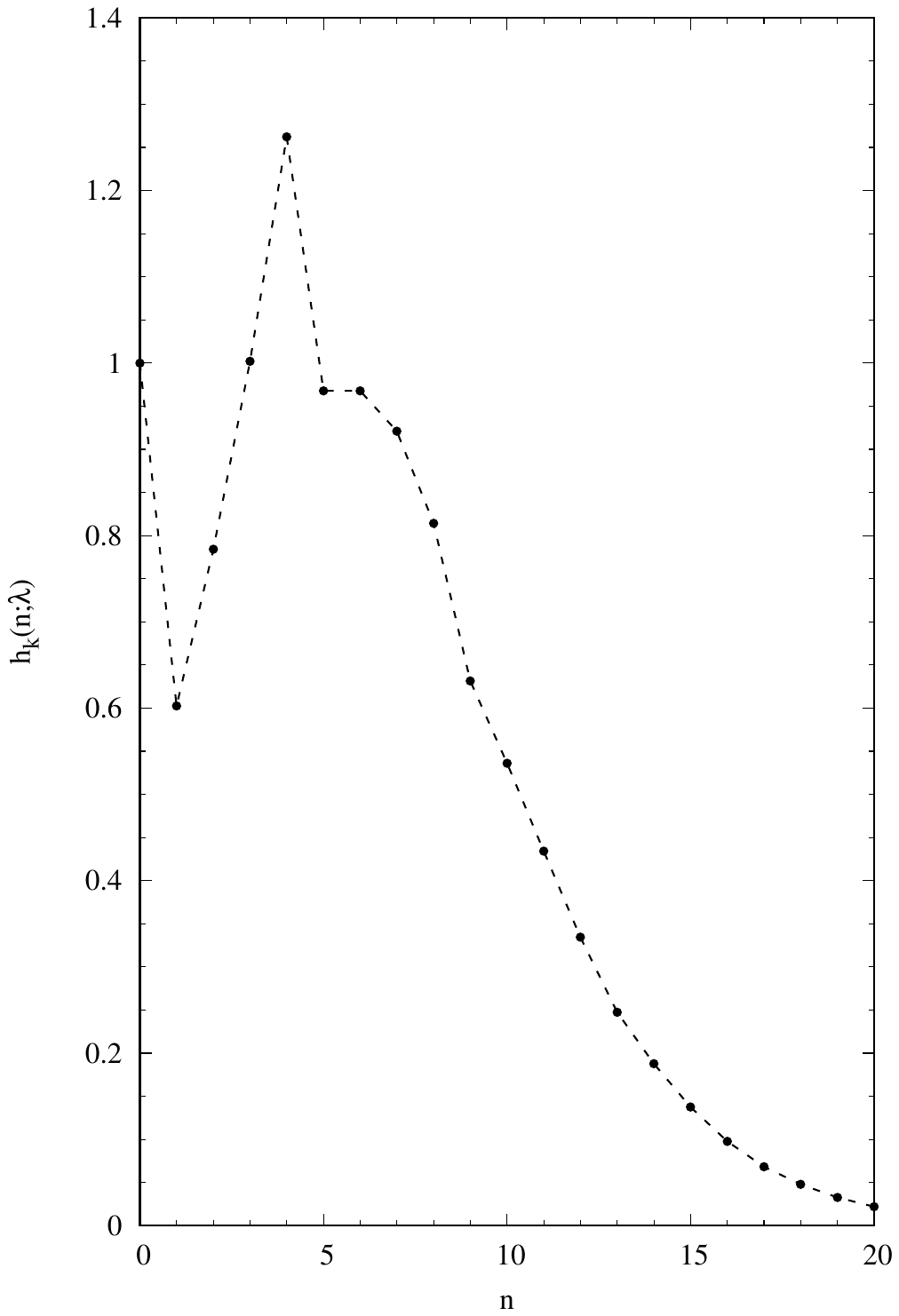}
\caption{\small
  \label{fig:graph_hist_k4}
  Histogram plot of $h_k(n;\lambda)\;(=p_n)$ of the Poisson distribution of order $4$ and $\lambda=0.6026076$.}
\end{figure}

\newpage
\begin{figure}[!htb]
\centering
\includegraphics[width=0.75\textwidth]{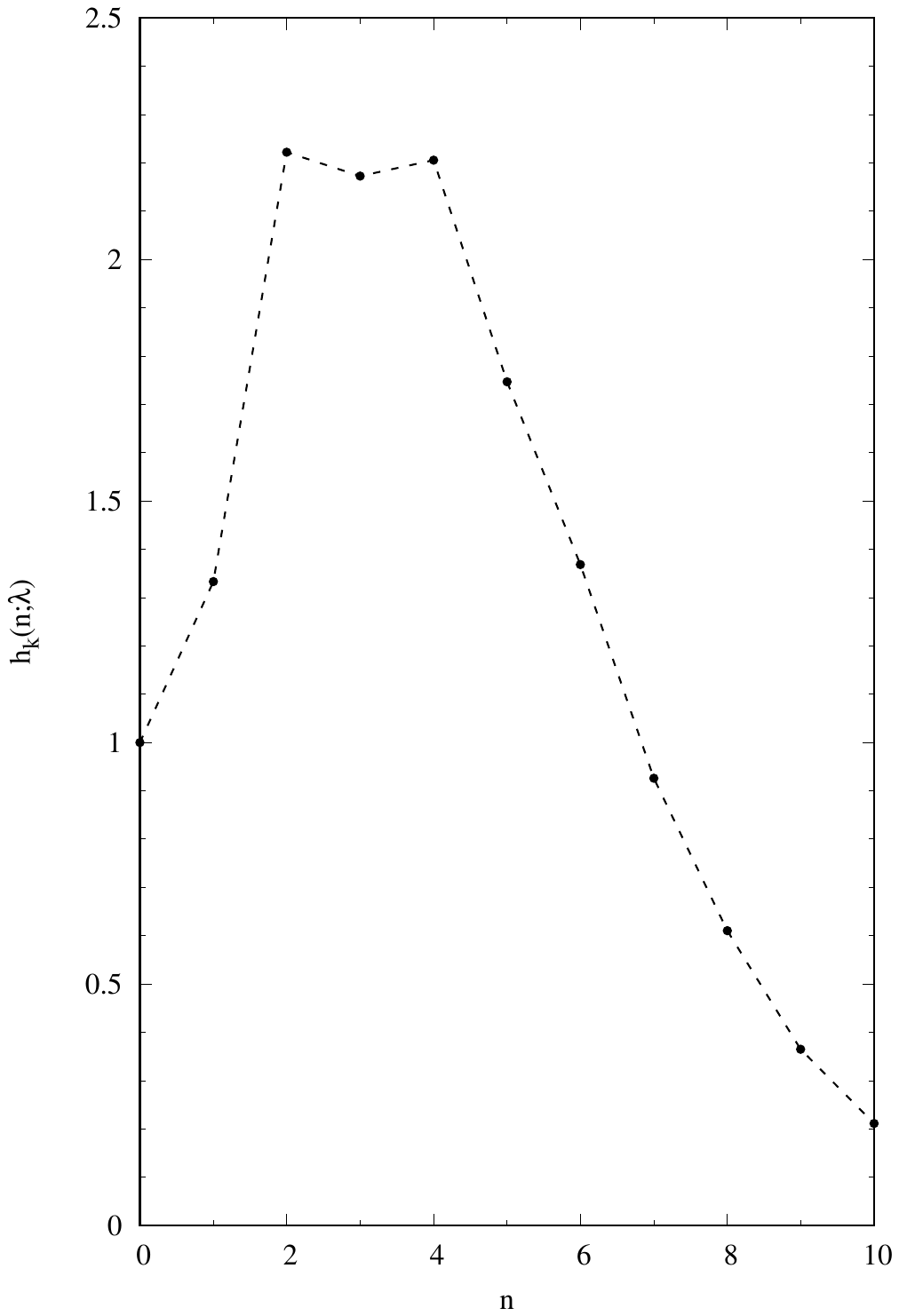}
\caption{\small
  \label{fig:graph_hist_k2}
  Histogram plot of $h_k(n;\lambda)\;(=p_n)$ of the Poisson distribution of order $2$ and $\lambda=4/3$.}
\end{figure}

\newpage
\begin{figure}[!htb]
\centering
\includegraphics[width=0.75\textwidth]{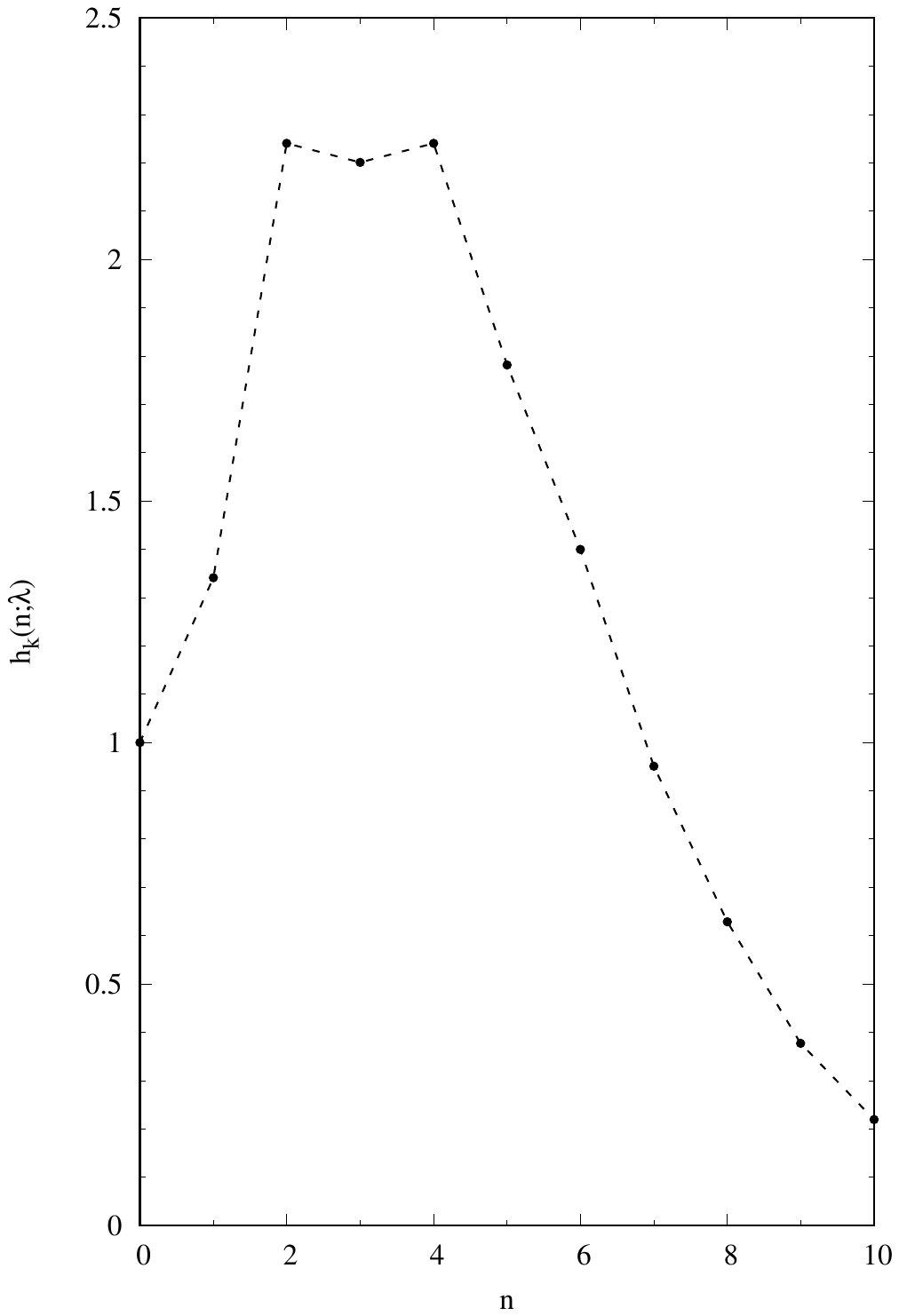}
\caption{\small
  \label{fig:graph_hist_k2_bimodal}
  Histogram plot of $h_k(n;\lambda)\;(=p_n)$ of the Poisson distribution of order $2$ and $\lambda=4.02373/3$.}
\end{figure}


\newpage
\begin{thebibliography}{99}
\bibitem{Adelson1966}
R.M.~Adelson, ``Compound Poisson Distributions''
{\it Operational Research Quarterly {\bf 17}, 73--75 (1966).}
\bibitem{KwonPhilippou}
Y.~Kwon and A.N.~Philippou, ``The Modes of the Poisson Distribution of Order 3 and 4''
{\it Entropy {\bf 25}, 699 (2023).}
\bibitem{Mane_Poisson_k_CC23_6}
S.R.~Mane, ``Structure of the probability mass function of the Poisson distribution of order $k$''
{\it http://arxiv.org/abs/2309.13493 [math.PR] (2023).}
\bibitem{KostadinovaMinkova2013}
K.Y.~Kostadinova and L.D.~Minkova, ``On the Poisson process of order $k$''
{\it Pliska Stud. Math. Bulgar. {\bf 22}, 117--128 (2013).}
\bibitem{PhilippouFibQ}
A.N.~Philippou, ``{\sc a note on the modes of the poisson distribution of order $k$}''
{\it Fibonacci Quarterly {\bf 52}, 203--205 (2014).}
\bibitem{Mane_Poisson_k_CC23_3}
S.R.~Mane, ``Asymptotic results for the Poisson distribution of order $k$''
{\it arXiv:2309.05190 [math.PR] (2023).}
\bibitem{GeorghiouPhilippouSaghafi}
C.~Georghiou, A.N.~Philippou and A.~Saghafi, ``On the Modes of the Poisson Distribution of Order $k$''
\bibitem{Mane_Poisson_k_CC23_5}
S.R.~Mane, ``First double mode of the Poisson distribution of order $k$''
{\it arXiv:2309.09278 [math.PR] (2023).}
\bibitem{PhilippouMeanVar}
A.N.~Philippou, ``Poisson and compound Poisson distributions of order $k$ and some of their properties''
\end{thebibliography}
\end{document}